\documentclass[12pt]{article}
\usepackage{eulervm} \usepackage[T1]{fontenc}
\usepackage{amsmath,amsthm,latexsym,amsfonts,amssymb,mathrsfs}
\usepackage{graphicx,verbatim}

\numberwithin{equation}{section}
\newtheorem{theorem}{Theorem}[section]
\newtheorem{lemma}[theorem]{Lemma}
\newtheorem{corollary}[theorem]{Corollary}

\newcommand{\err}{\mathsf{e}}
\newcommand{\LB}{\triangle^*}
\newcommand{\T}{\mathcal{T}}
\newcommand{\iprod}[1]{\langle#1\rangle}
\newcommand{\Sphere}{\mathbb{S}}
\newcommand{\Hilb}{\mathcal{H}}
\newcommand{\E}{\mathcal{E}}
\newcommand{\F}{\mathcal{F}}
\newcommand{\C}{\mathbb{C}}
\newcommand{\QQ}{\mathcal{Q}}
\newcommand{\R}{{\mathbb R}}

\newcommand{\Lapl}{\mathcal{L}}
\newcommand{\vecspan}{\operatorname{span}}
\newcommand{\grad}{\operatorname{grad}}
\newcommand{\vecG}{\boldsymbol{G}}
\newcommand{\vecU}{\boldsymbol{U}}

\author{Q.~T.~Le Gia and  William McLean\thanks{School 
     of Mathematics and Statistics,
     University of New South Wales, 
     Sydney, Australia.
     Email:{\tt qlegia@unsw.edu.au}, {\tt w.mclean@unsw.edu.au}}}
\title{Solving parabolic equations on the unit sphere via Laplace transforms
and radial basis functions}
\begin{document}
\maketitle
\begin{abstract}
We propose a method to construct numerical solutions of parabolic equations 
on the unit sphere. The time discretization uses Laplace transforms and 
quadrature. The spatial approximation of the solution employs radial 
basis functions restricted to the sphere. The method allows us 
to construct high accuracy numerical solutions in parallel.
We establish $L_2$ error estimates for smooth and nonsmooth initial data,
and describe some numerical experiments.
\end{abstract}

\vspace{0.5cm}


\vspace{0.5cm}


\section{Introduction}
We consider the initial-value problem 
\begin{equation}\label{equ:heat}
\partial_t u + A u = f(t), \quad\text{for $t>0$,}\quad\text{with $u(0)=u_0$,}
\end{equation}
where $\partial_t = \partial/\partial t$ and $A$ is a linear, self-adjoint,
positive-semidefinite, second-order elliptic partial differential operator 
on the unit sphere.  In our standard example, $-A$ is the 
Laplace--Beltrami operator.  The source term~$f(t)$ may depend on
the spatial variables but we suppress this dependence in our notation,
viewing $f(t)$ as an element of a function space on the sphere.

Instead of using time stepping for the numerical solution, as was done
previously~\cite{LeG05a}, our approach is to represent the solution 
of~\eqref{equ:heat}
as an inverse Laplace transform, which is then approximated by quadrature.
Developed first for parabolic problems by Sheen, Sloan and 
Thom\'ee~\cite{SheSloTho99}, such an approach is also effective for 
some evolution equations with memory~\cite{McLTho10}.  These and related
papers have discussed thoroughly the time discretization, but for
the space discretization have considered only piecewise linear finite 
elements on a bounded domain in~$\R^n$.  Here, we propose instead a
space discretization using spherical radial basis functions (SRBFs), 
which are convenient for parabolic problems on Riemannian surfaces 
such as the unit sphere~$\Sphere^n=\{\,x\in\R^{n+1}:|x|=1\,\}$.

Denoting the Laplace transform of~$u$ with respect to~$t$ by
\begin{equation}\label{equ:Lap def}
  \hat u(z) = \Lapl\{u(t)\} :=\int_0^\infty e^{-zt} u(t) dt,
\end{equation}
we find that the solution of~\eqref{equ:heat} formally satisfies
\begin{equation}\label{equ:Lap of heat}
   (zI + A)\hat u(z) = g(z) := u_0 + \hat f(z),
\end{equation}
where $I$ denotes the identity operator.  The spectrum of~$A$ is
a subset of the half-line~$[0,\infty)$, so if $z\notin(-\infty,0]$
and if the Laplace transform~$\hat f(z)$ exists, then
\begin{equation}\label{equ:uhat}
  \hat u(z) = (zI+A)^{-1} g(z).
\end{equation}
When $\hat f(z)$ is analytic and bounded for~$\Re z>0$, the 
solution~$u(t)$ can be recovered via the Laplace inversion formula
\begin{equation}\label{equ:inv Lap}
u(t) = \frac{1}{2\pi i} \int_{\Gamma_0} e^{zt} \hat u(z) dz,  
\quad\text{for $t>0$,}
\end{equation}
where $\Gamma_0$ is the contour~$\Re z=\omega$, for any $\omega>0$,
with $\Im z$ increasing.

Section~\ref{sec: prelim} summarizes some technical results and 
assumptions needed for our subsequent analysis.
In Section~\ref{sec: discrete} we describe the time discretization
and quote a known error estimate (Theorem~\ref{thm:time dis}), after
which we introduce the space discretization using SRBFs.  The heart
of the paper is Section~\ref{sec: error}, where we prove two error
bounds for the space discretization by adapting the analysis of
Thom\'ee~\cite{Tho97} for a finite element approximation of the
heat equation on a domain in~$\R^n$.  The first bound 
(Theorem~\ref{thm:space dis}) requires some spatial regularity of 
$u_0$~and $f$, and is proved by estimating a contour integral.  
The second bound is proved by an energy argument, and assumes $f\equiv0$ 
but allows nonsmooth initial data~$u_0\in L_2(\Sphere^n)$.  Both bounds
include a factor that blows up as~$t\to0$.  Finally, 
Section~\ref{sec: num exp} describes the results of some numerical
experiments.
\section{Preliminaries}\label{sec: prelim}
\subsection{Resolvent estimates}
We now view $A$ as an abstract, densely defined, self-adjoint and
positive-semidefinite linear operator on a complex Hilbert space~$\Hilb$.
Assume further that $(I+A)^{-1}:\Hilb\to\Hilb$ is compact, so $A$ has a
discrete spectrum, and order the eigenvalues 
$0\le\lambda_1\le\lambda_2\le\cdots$.  Note that $\lambda_j\to\infty$ 
as~$j\to\infty$ if $\Hilb$ is infinite dimensional.

For any~$\varphi>0$, the spectrum of~$A$ is a subset of a closed 
sector in the complex plane~$\C$, 
\[
\Sigma_{\varphi}:= \{\,z \ne 0: |\arg z| \le \varphi\, \} \cup \{0\},
	\quad\text{with $0 < \varphi < \pi/2$.}
\]
In addition, there is a constant~$C>0$ such that $A$ satisfies the 
resolvent estimate $\|(zI-A)^{-1}\| \le C|z|^{-1}$ 
for $z \in \overline{\C\setminus \Sigma_{\varphi}}$, or, equivalently,
\begin{equation}\label{equ:res est}
 \|(zI+A)^{-1}\| \le C|z|^{-1},
	\quad\text{for $z \in \Sigma_{\pi-\varphi}$,}
\end{equation}
where~$\|\cdot\|$ denotes the operator norm induced by the norm
in~$\Hilb$.  

\subsection{Sobolev spaces on the unit sphere}
Denote the inner product in~$\Hilb=L_2(\Sphere^n)$ by
\[
\iprod{v,w}:=\int_{\Sphere^n}vw\,dS,
\]
where $dS$ is the surface measure on the unit sphere, and
denote the measure of the whole sphere by~$\omega_n$
(so, for example, $\omega_2=4\pi$).
Recall~\cite{Mul66} that a spherical harmonic is the restriction
to~$\Sphere^n$ of a homogeneous polynomial~$Y(x)$ in~$\R^{n+1}$
satisfying $\triangle Y(x) = 0$, where $\triangle$ is the Laplacian
operator in $\R^{n+1}$.  The space of spherical
harmonics of degree~$\ell$, denoted by~$\Hilb_\ell$, has 
dimension~$N(n,\ell):=\dim\Hilb_\ell$, given by
\[
N(n,0)=1\quad\text{and}\quad
  N(n,\ell) = \frac{(2\ell+n-1)(\ell+n-2)!}
      {\ell!(n-1)!} \quad\text{for $\ell\ge 1$.}
\]
In the usual way, we construct an orthonormal basis
$\{\,Y_{\ell k}:1\le k\le N(n,\ell)\,\}$ for~$\Hilb_\ell$, so that
$\iprod{Y_{\ell k},Y _{\ell' k'}}=\delta_{\ell\ell'}\delta_{kk'}$.

The Laplace--Beltrami operator~$\LB$ on~$\Sphere^n$ may be defined in
terms of the Laplacian~$\triangle$ on~$\R^{n+1}$ by
\begin{equation}\label{eq: LB}
\LB v=\triangle\check v|_{\Sphere^n}\quad\text{where}\quad
\check v(x)=v(x/|x|).
\end{equation}
The spherical harmonics are eigenfunctions of~$\LB$, satisfying
\[
-\LB Y_{\ell k}=\lambda_\ell Y_{\ell k}
\quad\text{where}\quad\lambda_\ell=\ell(\ell+n-1),
\]
for $1\le k\le N(n,\ell)$ and $\ell\in\{0,1,2,\ldots\}$.  
Every function $v \in L_2(\Sphere^n)$ can be expanded in a generalized
Fourier series
\[
v  = \sum_{\ell=0}^\infty\sum_{k=1}^{N(n,\ell)} 
	\hat v_{\ell k} Y_{\ell k}\quad\text{where}\quad
	\hat v_{\ell k} = \iprod{v,Y_{\ell k}},
\]
and for $\sigma\in\R$ we can characterize the Sobolev space on the unit 
sphere, $H^\sigma=H^\sigma(\Sphere^n)$, in terms of the generalized Fourier 
coefficients: $v\in H^\sigma$ if and only if the norm defined by
\begin{equation}\label{eq: H^sigma norm}
\|v\|^2_{H^\sigma}:=\bigl\|(I-\LB)^{\sigma/2}v\bigr\|^2
	=\sum_{\ell=0}^{\infty}(1+\lambda_\ell)^\sigma
		\sum_{k=1}^{N(n,\ell)}|\hat{v}_{\ell k}|^2
\end{equation}
is finite.  We also define the subspace of functions with mean zero,
\[
H^\sigma_0 = H^\sigma_0(\Sphere^n):= 
  \bigl\{\, v \in H^\sigma(\Sphere^n): \int_{\Sphere^n}v\,dS=0\,\bigr\};
\]
since $Y_{01}=1/\sqrt{\omega_n}$ is constant, we see that 
$v\in H^\sigma$ belongs to~$H^\sigma_0$ if and only if $\hat v_{01}=0$.
\subsection{Positive definite kernels on the unit sphere}
A continuous function $\Phi :\Sphere^n\times\Sphere^n \to \R$ is called a
\emph{positive definite kernel}~\cite{Sch42,XuChe92} 
on~$\Sphere^n$ if it satisfies the following two conditions:
\begin{itemize}
\item[(i)] $\Phi(x,y) = \Phi(y,x)$ for all
$x$, $y\in\Sphere^n$;
\item[(ii)] for any set of distinct scattered points
$\{y_1,y_2,\ldots,y_K\} \subset \Sphere^n$, the symmetric
$K\times K$~matrix $[\Phi(y_i,y_j)]$ is positive semi-definite.
\end{itemize}
We call $\Phi$ \emph{strictly} positive definite if the matrix is strictly
positive definite.

We will work with a kernel~$\Phi$ defined in terms of a univariate 
function~$\phi:[-1,1] \rightarrow \R$ by
\begin{equation}\label{equ:Phi res}
\Phi(x,y)=\phi(x\cdot y)
\quad\text{for all $x$, $y\in\Sphere^n$,}
\end{equation}
where $x\cdot y$ denotes the Euclidean inner product of $x$~and $y$.
Following M\"uller~\cite{Mul66}, let $P_\ell(t)$ denote the Legendre 
polynomial of degree~$\ell$ for~$\R^{n+1}$, and expand $\phi(t)$ in a
Fourier--Legendre series
\begin{equation}\label{equ:defphi}
\phi(t) = \frac{1}{\omega_n}
          \sum_{\ell=0}^\infty
           N(n,\ell)\, a_\ell P_\ell(t).
\end{equation}
Due to the addition formula for spherical harmonics~\cite[Page~10]{Mul66},
\[
\sum_{k=1}^{N(n,\ell)}
    Y_{\ell k}(x) Y_{\ell k}(y) =
         \frac{N(n,\ell)}{\omega_n} P_\ell(x \cdot y),
\]
the kernel $\Phi$ can be represented as
\begin{equation}\label{def_Phi}
  \Phi(x,y) =\sum_{\ell=0}^\infty 
       \sum_{k=1}^{N(n,\ell)}a_\ell Y_{\ell k}(x)Y_{\ell k}(y),
\end{equation}
and since $P_\ell(1)=1$ we find that
\begin{equation}\label{eq: Phi H^tau}
\|\Phi(x,\cdot)\|_{H^\tau}^2=\frac{1}{\omega_n}\sum_{\ell=0}^\infty
	(1+\lambda_\ell)^\tau a_\ell^2 N(n,\ell),
	\quad\text{for all $x\in\Sphere^n$.}
\end{equation}

Chen et al.~\cite{CheMenSun03} proved that
the kernel $\Phi$ is strictly positive definite if and only if
$a_\ell\ge 0$ for all $\ell\ge 0$ and
$a_\ell> 0$ for infinitely many even values of~$\ell$ and
infinitely many odd values of~$\ell$; see also Schoenberg~\cite{Sch42}
and Xu and Cheney~\cite{XuChe92}.
Here, we assume there is a $\tau>n/2$ and positive constants $c$~and $C$
such that
\begin{equation}\label{cond:what}
c(1+\lambda_\ell)^{-\tau}\le a_\ell\le C(1+\lambda_\ell)^{-\tau}, 
\quad \text{for all $\ell\ge 0$.}
\end{equation}
Hence, $\Phi$ is strictly positive definite and, since 
$N(n,\ell)=O(\ell^{n-1})$ as~$\ell\to\infty$, the 
sum~\eqref{eq: Phi H^tau} is finite so, for each 
fixed~$x\in\Sphere^n$, the function~$y\mapsto\Phi(x, y)$ belongs 
to~$H^\tau(\Sphere^n)$. Moreover, this function is continuous
by the Sobolev imbedding theorem.

\section{The discrete problem}\label{sec: discrete}
Choose an angle~$\beta\in(\pi/2,\pi-\varphi)$ and let $\Gamma$ be any curve 
in the interior of the sector $\Sigma_\beta$ which is homotopic to the 
line~$\Gamma_0$ appearing in the Laplace inversion 
formula~\eqref{equ:inv Lap}.  Deforming the contour of 
integration in~\eqref{equ:inv Lap}, we may then write
\begin{equation}\label{equ:u Gamma}
 u(t) = \frac{1}{2\pi i} \int_\Gamma e^{zt} \hat u(z)dz,
\end{equation}
assuming that $\hat f(z)$ is analytic on and to the right of~$\Gamma$.

By taking $f\equiv 0$ in~\eqref{equ:heat}, so that $g(z) = u_0$ in 
\eqref{equ:Lap of heat}, we see that the solution operator for
the homogeneous problem has the integral representation
\begin{equation}\label{equ:calE}
 \E(t) u_0 = \frac{1}{2\pi i} \int_{\Gamma} e^{zt} \hat\E(z) u_0 dz,
\quad\text{where $\hat\E(z)=(zI+A)^{-1}$.}
\end{equation}
For the inhomogeneous case, the inverse Laplace transform of
$\hat\E(z)\hat f(z)$ is the convolution of $\E(t)$ and $f(t)$,
giving the Duhamel formula
\begin{equation}\label{equ:Duhamel}
u(t) =\E(t)u_0 + \int_{0}^t \E(t-s)f(s)ds.
\end{equation}
A standard energy argument shows that $\|\E(t)u_0\|\le\|u_0\|$ for
all $t\ge0$, so the continuous problem~\eqref{equ:heat} is stable in 
the sense that
\[
\|u(t)\| \le \|u_0\|+\int_0^t\|f(s)\|ds,
\quad\text{for $t \ge 0$.}
\]

For our numerical methods we choose $\Gamma$ to be the curve with parametric 
representation
\begin{equation}\label{equ:z xi}
z(\xi):=  \omega+\lambda\bigl(1-\sin(\delta-i\xi)\bigr), 
\quad \mbox{ for } \xi \in \R,
\end{equation}
where the constants $\omega$, $\lambda$ and $\delta$ satisfy
\begin{equation}\label{equ:conds}
\omega>0,\quad \lambda >0 \quad \text{and} \quad 0<\delta<\beta-\pi/2.
\end{equation}
Writing $z=x+iy$, we find that $\Gamma$ is the left branch of the
hyperbola
\begin{equation}\label{eq: hyperbola}
\biggl( \frac{x-\omega-\lambda}{\lambda\sin\delta} \biggr)^2
-\biggl(\frac{y}{\lambda\cos\delta}\biggr)^2 = 1,
\end{equation}
which cuts the real axis at the point $z=\omega+\lambda(1-\sin\delta)$
and has asymptotes $y=\pm(x-\omega-\lambda)\cot \delta$. Thus, the
conditions~\eqref{equ:conds} ensure that $\Gamma$ lies in the
sector~$\Sigma_{\beta}^{\omega} := \omega+\Sigma_{\beta} 
\subset \Sigma_\beta$, and crosses into the left half-plane.

We use \eqref{equ:z xi} in~\eqref{equ:u Gamma} to represent
$u(t)$ as an integral with respect to~$\xi$,
\begin{equation}\label{equ:inv Lap v}
u(t)= \frac{1}{2\pi i}\int_{-\infty}^{\infty}
	e^{z(\xi)t} w(z(\xi)) z'(\xi)\, d\xi.
\end{equation}
Since $|e^{z(\xi)t}|=e^{\Re z(\xi)t}
=e^{\omega t} e^{\lambda t(1-\sin\delta\cosh\xi)}$, the integrand
exhibits a double exponential decay as $|\xi| \to \infty$, for
any fixed~$t>0$.

\subsection{Time discretization}
We choose a quadrature step~$k$, put
\[
\xi_j := jk,\quad z_j:=z(\xi_j), \quad z_j':=z'(\xi_j), 
\]
and apply an equal weight rule to the integral~\eqref{equ:inv Lap v}
to obtain an approximate solution 
\begin{equation}\label{equ:UN}
U_N(t) := \frac{k}{2\pi i} \sum_{j=-N}^N e^{z_j t} \hat u(z_j) z'_j.
\end{equation}
In view of~\eqref{equ:u Gamma}, to compute $U_N(t)$ we must solve 
the $2N+1$~equations
\begin{equation}\label{equ:elliptic}
 (z_j I + A)\hat u(z_j) = g(z_j), \quad \mbox{ for } |j| \le N.
\end{equation}
These equations are independent and hence may be solved in parallel.
Notice that the $\hat u(z_j)$ determine the approximate 
solution~\eqref{equ:UN} 
for all $t>0$ and that the numerical solution \eqref{equ:UN} depends on the 
choice of the curve~$\Gamma$, even though the 
representation~\eqref{equ:u Gamma} does not.  However, we will see that 
a given $\Gamma$~and $k$ yield an accurate 
approximation~$U_N(t) \approx u(t)$ only for~$t$ at a particular time scale.

The parametric representation~\eqref{equ:z xi} of~$\Gamma$ extends to a 
conformal mapping
\begin{equation}\label{equ:conformal}
 z = \Psi(\zeta) = \omega + \lambda\bigl(1 - \sin(\delta - i \zeta)\bigr),
\end{equation}
which, for~$r>0$, transforms the 
strip~$Y_r:= \{ \zeta: |\Im \zeta| \le r \}$ 
onto the set~$S_r:= \{ \Psi(\zeta): \zeta \in Y_r \} \supset \Gamma$.
In fact, $\Psi$ maps the line~$\Im\zeta = \eta$ to the left branch
of a hyperbola given by~\eqref{eq: hyperbola} with~$\delta$ replaced 
by~$\delta+\eta$.  Thus, $S_r$ is bounded by the left branches
of the hyperbolas corresponding to $\Im \zeta = r$~and $\Im \zeta = -r$. 
To ensure that $S_r \subset \Sigma_\beta^\omega$ and that 
$\Re z \to -\infty$ if $|z| \rightarrow \infty$ with~$z \in S_r$, we 
require $ 0 <\delta-r < \delta +r < \beta - \pi/2$, or equivalently that
\begin{equation}\label{cond:r}
0 < r < \min(\delta, \beta - \pi/2 - \delta).
\end{equation}


We introduce the notation
\[
 \|g\|_{X,Z} := \sup_{z \in Z} \|g(z)\|_X, 
	\quad\text{ for $X \subseteq \Hilb$ and $Z \subseteq \C$,}
\]
abbreviated by $\|g\|_Z$ if $X=\Hilb$, and put 
$\lg(s)=\max\bigl(1,\log(1/s)\bigr)$.

\begin{theorem}\label{thm:time dis}
Let $u$ be the solution of~\eqref{equ:heat}, 
with~$\hat f$ bounded and analytic 
in~$\Sigma_{\beta}^\omega$, and fix a time scale~$T>0$.  Let
$0 < \theta <1$ and define $b>0$ by~$\cosh b = 4/(\theta\sin \delta)$, 
let $r$ satisfy \eqref{cond:r} so that 
$\Gamma \subset S_r \subset \Sigma_{\beta}^\omega$, and put
$\lambda = \pi r\theta N/(bT)$. Then
the approximate solution $U_N(t)$ defined by~\eqref{equ:UN} 
with~$k=b/N \le 2\pi r\log 2$ satisfies
\[
\|U_N(t)-u(t)\| \le 
C e^{\omega t} \lg(\rho_r N) e^{-\mu N}
\bigl(\|u_0\| + \| \hat f \|_{\Sigma^\omega_{\beta}}\bigr),
\quad\text{ for $T/2\le t \le2T$,}
\]
where $\mu =2\pi r (1-\theta)/b$, $\rho_r = \pi r\theta \sin(\delta-r)/(2b)$
and $C = C_{\delta,r,\beta}$.
\end{theorem}
\begin{proof}
See McLean and Thom\'ee~\cite[Theorem~3.1]{McLTho10}.
\end{proof}
\subsection{Galerkin approximation by SRBFs}
Given a suitable set of 
points~$X=\{x_1,x_2,\ldots,x_K\}\subseteq\Sphere^n$ and 
a strictly positive definite kernel~$\Phi(x,y)$, we define
the spherical radial basis functions~$\Phi_p(x):=\Phi(x_p,x)$ 
for~$1\le p\le K$.  Recall that our assumption~\eqref{cond:what} ensures 
$\Phi_p\in H^\tau$ with~$\tau>n/2\ge1$; thus
\[
S_h := \vecspan \{\, \Phi_p: 1\le p\le K\, \}\subseteq H^1.
\]
The uniformity of the set $X$ is measured by its mesh norm~$h_X$ 
and its separation radius~$q_X$, defined by
\[
h=h_X := \sup_{y \in \Sphere^n} 
  \min_{x \in X} \cos^{-1}(y\cdot x)
  \quad\text{and}\quad
q=q_X := \frac{1}{2} \min_{\substack{x \ne y\\ x,y\in X}} \cos^{-1}(y\cdot x).
\]
In words, $h_X$ is the maximum geodesic distance from a point 
on~$\Sphere^n$ to the nearest point of~$X$.  For our convergence
analysis, we require that the family of point sets~$\{X\}$ has a
bounded mesh ratio:
\begin{equation}\label{eq: mesh ratio}
h_X\le C q_X.
\end{equation}

Associated with the second-order, partial differential differential
operator~$A$ is a bounded sesquilinear form~$a:H^1\times H^1\to\C$
defined by
\[
a(u,v)=\iprod{Au,v}\quad\text{for $u$, $v\in H^1$.}
\]
For example, if $A=-\LB$ then $a(u,v)=\iprod{\grad u,\grad v}$ where
$\grad$ is the surface gradient.  The mild solution 
$u:[0,\infty)\to L_2(\Sphere^n)$ of~\eqref{equ:heat} satisfies 
\[
\iprod{\partial_tu,v}+a(u,v)=\iprod{f(t),v}
	\quad\text{for $t>0$ and all $v\in H^1$,}
\]
with $u(0)=u_0$, and we define a semidiscrete 
solution~$u_h:[0,\infty)\to S_h$ of~\eqref{equ:heat} by
\begin{equation}\label{eq: uh weak}
\iprod{\partial_tu_h,\chi}+a(u_h,\chi)=\iprod{f(t),\chi}
	\quad\text{for all $\chi\in S_h$,}
\end{equation}
with $u_h(0)=u_{0h}\approx u_0$ for a suitable~$u_{0h}\in S_h$.

The Laplace transform of~$u$ at~$z_j$ is the weak 
solution~$\hat u(z_j)\in H^1$ of~\eqref{equ:elliptic}, that is,
\[
z_j\iprod{\hat u(z_j),v}+a\bigl(\hat u(z_j),v\bigr)=\iprod{g(z_j),v}
	\quad\text{for all $v\in H^1$,}
\]
and the Laplace transform of the semidiscrete solution,
$\hat u_h(z_j)\in S_h$, satisfies
\begin{equation}\label{eq: Galerkin}
z_j\iprod{\hat u_h(z_j),\chi}+a\bigl(\hat u_h(z_j),\chi\bigr)
	=\iprod{g_h(z_j),\chi}\quad\text{for all $\chi\in S_h$,}
\end{equation}
where $g_h(z)=u_{0h}+P_h\hat f(z)\in S _h$ and $P_h$ denotes the
orthogonal projector from~$L_2(\Sphere^n)$ onto~$S_h$.
Thus, we can view $\hat u_h(z_j)$ as a Galerkin approximation 
to~$\hat u(z_j)$.
Concretely, to compute $\hat u_h(z)=\sum_{p=1}^K\hat U_p(z)\Phi_p$ 
we form the $K \times K$ matrices $B$~and $S$, with entries
\begin{equation}\label{equ:B S}
B_{pq} =  \iprod{\Phi_p,\Phi_q}
\quad\text{and}\quad
S_{pq}  = a(\Phi_p,\Phi_q), 
\end{equation}
form the load vector~$\vecG(z) \in \C^K$ with 
components~$G_p(z) = \iprod{g_h(z),\Phi_p}$, and then solve the 
$K\times K$ complex linear system
\begin{equation}\label{equ:linear system}
 (z_j B + S)\hat\vecU(z_j) = \vecG(z_j),
\end{equation}
to obtain the solution vector $\hat\vecU(z) \in \C^K$ with 
components~$\hat U_p(z)$.
In contrast to finite element mass and stiffness matrices, $B$~and $S$
are not sparse because the SRBFs have large supports.
\subsection{Fully-discrete solution}
Combining the time and space discretizations, we arrive at a fully-discrete
solution
\begin{equation}\label{equ:full}
U_{N,h}(t)=\frac{k}{2\pi i}\sum_{j=-N}^N e^{z_jt}\hat u_h(z_j)z'_j,
\end{equation}
whose evaluation requires that we solve the linear
system~\eqref{equ:linear system} at each of the $2N+1$ quadrature
points~$z_j$.  (In practice, we also use quadratures for the
integrations over~$\Sphere^n$ that are needed to compute $B_{pq}$, $S_{pq}$
and $G_p(z)$, but for our analysis we assume that these quantities
are computed exactly.) The elliptic differential operator~$A$ induces a
discrete operator~$A_h:S_h\to S_h$, defined by
\begin{equation}\label{eq: Ah}
\iprod{A_h\psi,\chi} = a(\psi,\chi), 
	\quad\text{for $\psi$, $\chi \in S_h$,}
\end{equation}
and the Galerkin equations~\eqref{eq: Galerkin} are equivalent to
\begin{equation}\label{eq: uh hat}
(z_jI+A_h)\hat u_h(z_j)=g_h(z_j).
\end{equation}
If we choose $u_{0h}=P_hu_0$ then $g_h(z_j)=P_hg(z_j)$ and by
taking $\Hilb=S_h$ equipped with the $L_2$-norm, we can apply
Theorem~\ref{thm:time dis} to~$A_h$ and deduce that
\begin{equation}\label{eq: h error}
\|U_{N,h}(t)-u_h(t)\| \le 
C e^{\omega t} \lg(\rho_r N) e^{-\mu N}
\bigl(\|u_0\| + \| \hat f \|_{\Sigma^\omega_{\beta}}\bigr),
\quad\text{ for $T/2\le t \le2T$.}
\end{equation}
Since the triangle inequality gives
\begin{equation}\label{eq: triangle ineq}
\|U_{N,h}(t)-u(t)\|\le\|U_{N,h}(t)-u_h(t)\|+\|u_h(t)-u(t)\|,
\end{equation}
to estimate the error in~$U_{N,h}$ it now suffices to estimate
the error in the semidiscrete approximation~$u_h(t)$.
\section{Error analysis of the spatial discretization}\label{sec: error}
We assume now that $A=-\LB$.  Since $\lambda_0=0$ but
$\lambda_\ell\ge\lambda_1=n$ for all $\ell\ge1$, we see that 
$1+\lambda_\ell\le(1+n^{-1})\lambda_\ell$ for all~$\ell\ge1$.  Hence, 
the sesquilinear form~$a$ is coercive on~$H^1_0$, that is,
\begin{equation}\label{eq: a coercive}
a(v,v)\ge \frac{\|v\|_{H^1}^2}{1+n^{-1}}
	\quad\text{if $v\in H^1$ and $\hat v_{10}=\int_{\Sphere^n}v\,dS=0$.}
\end{equation}
Our analysis follows Thom\'ee~\cite[Chapter~3]{Tho97}, with $\LB$
in place of the Laplacian (with homogeneous Dirichlet boundary conditions).
Some technical modifications are needed, however, because $\LB$ has a
zero eigenvalue.
\subsection{Approximation by SRBFs}
We will use the following estimate for the best approximation by
SRBFs.

\begin{theorem}\label{thm: best approx}
Assume that the Fourier--Legendre coefficients in the 
expansion~\eqref{equ:defphi} satisfy \eqref{cond:what} with~$\tau>n/2$, 
so that $S_h\subseteq H^\tau(\Sphere^n)$.  For any real $q$~and $\nu$
satisfying $q\le\nu\le2\tau$ and $q\le\tau$, if $v\in H^\nu$
then there exists $\chi\in S_h$ such that
\[
\|\chi-v\|_{H^q}\le C h_X^{\nu-q}\|v\|_{H^\nu}.
\]
\end{theorem}
\begin{proof}
See Tran et al. \cite[Theorem~3.2]{TranPham2008}~or
\cite[Theorem~3.7 and Remark~5.1]{TranEtAl2009}, and note our
assumption~\eqref{eq: mesh ratio}.
\end{proof}

In the special case~$q=0$, the estimate must hold for~$\chi=P_hv$,
giving the following result.

\begin{corollary}\label{cor: I-Ph}
The $L_2$-projection of~$v$ onto~$S_h$ has the approximation property
\[
\|v-P_hv\|\le Ch_X^\nu\|v\|_{H^\nu}\quad\text{for $0\le\nu\le2\tau$.}
\]
\end{corollary}

For our error analysis, we also use the Ritz 
projector~$R_h:H^1(\Sphere^n)\to S_h$ determined by the
sesquilinear form
\[
a_1(u,v)=a(u,v)+\iprod{u,v}\quad\text{for $u$, $v\in H^1$.}
\]
We see from~\eqref{eq: a coercive} that $a_1$ is coercive on~$H^1$;
in fact, $a_1(v,v)=\|v\|_{H^1}^2$.  Thus, $R_hv\in S_h$ is well-defined by
\begin{equation}\label{eq: Rh}
a_1(R_hv,\chi)=a_1(v,\chi)\quad\text{for all $\chi\in S_h$,}
\end{equation}
and the following error estimates hold using standard arguments.

\begin{theorem}\label{thm: Rh error}
If $v\in H^\nu$ and $1\le\nu\le2\tau$, then
\[
\|v-R_hv\|_{H^1}=\inf_{\chi\in S_h}\|v-\chi\|_{H^1}
	\le Ch_X^{\nu-1}\|v\|_{H^\nu}
\]
and
\[
\|v-R_hv\|\le Ch_X^\nu\|v\|_{H^\nu}.
\]
\end{theorem}
\begin{proof}
The definition~\eqref{eq: Rh} immediately implies the 
orthogonality property
\begin{equation}\label{eq: Rh orthog}
a_1(v-R_hv,\chi)=0\quad\text{for all $\chi\in S_h$,}
\end{equation}
so, because $a_1(v,v)=\|v\|_{H^1}^2$,
\begin{align*}
\|v-R_hv\|_{H^1}^2&=a_1(v-R_hv,v-R_hv)=a_1(v-R_hv,v-\chi)\\
	&\le\|v-R_hv\|_{H^1}\|v-\chi\|_{H^1},
\end{align*}
and thus $\|v-R_h v\|_{H^1}\le \|v-\chi\|_{H^1}$ for all $\chi\in S_h$.
The first claim now follows by Theorem~\ref{thm: best approx}.

A duality argument~\cite{NitSch74} yields the second claim.
Given~$v$ there is a unique~$u\in H^1$ satisfying
$(I+A)u=v-R_hv$, or equivalently (since $A$ is self-adjoint)
\[
a_1(w,u)=\iprod{w,v-R_hv}\quad\text{for all $w\in H^1$,}
\]
Taking $w=v-R_hv$ and applying \eqref{eq: Rh orthog}, 
we have for every~$\chi\in S_h$,
\begin{align*}
\iprod{v-R_hv,v-R_hv}&=a_1(v-R_hv,u)=a_1(v-R_hv,u-\chi)\\
	&\le \|v-R_hv\|_{H^1}\|u-\chi\|_{H^1}
	\le Ch^{\nu-1}\|v\|_{H^\nu}\|u-\chi\|_{H^1}.
\end{align*}
By Theorem~\ref{thm: best approx} with $q=1$ and 
$\nu=2\le2\tau$, there is a $\chi\in S_h$ such that
$\|u-\chi\|_{H^1}\le Ch\|u\|_{H^2}$, so 
\[
\|v-R_hv\|^2\le Ch^\nu\|v\|_{H^\nu}\|u\|_{H^2},
\]
and the result follows because $\|u\|_{H^2}=\|(I+A)u\|=\|v-R_hv\|$.
\end{proof}

\subsection{Contour integral estimate}
We see from \eqref{equ:Lap of heat}~and \eqref{eq: uh hat} that,
assuming $u_{0h}=P_hu_0$,
\[
\hat u(z)=(zI+A)^{-1}g(z)\quad\text{and}\quad
\hat u_h(z)=(zI+A_h)^{-1}P_hg(z),
\]
so
\[
\hat u_h(z)-\hat u(z)=G_h(z)g(z)
	\quad\text{where}\quad
G_h(z):=(zI+A_h)^{-1}P_h-(zI+A)^{-1}.
\]
Deforming the integration contour in the Laplace inversion formula
to~$\Gamma=\partial\Sigma^\omega_\beta$, we can represent the error
in the semidiscrete solution as follows:
\begin{equation}\label{eq: uh-u integral}
u_h(t)-u(t)=\frac{1}{2\pi i}\int_{\Gamma}e^{zt}G_h(z)g(z)\,dz.
\end{equation}
The next lemma allows us to estimate this integral.

\begin{lemma}\label{lem:Ghv}
If $0\le\nu\le2\tau$, then
\[
\|G_h(z) v\|\le C h_X^\nu \|v\|_{H^{\nu-2}}, 
\quad\text{for $z \in \Sigma_\beta^\omega$ and $v\in H^{\nu-2}$.}
\]
\end{lemma}
\begin{proof}
Recall that $\hat\E(z):=(zI+A)^{-1}$, and let $\hat\E_h(z):=(zI+A_h)^{-1}$. 
We split $G_h(z)$ into two terms,
\begin{equation}\label{eq: Gh split}
G_h(z)=(P_h-I)\hat\E(z)+\bigl[\hat\E_h(z)P_h-P_h\hat\E(z)\bigr].
\end{equation}
Since $A\hat\E(z)=(zI+A-zI)(zI+A)^{-1}=I-z(zI+A)^{-1}$, the resolvent
estimate~\eqref{equ:res est} shows that 
\[
\|A\hat\E(z)v\|\le C\|v\|\quad\text{for $z\in\Sigma^\omega_\beta$.}
\]
Moreover, since $(I+A)\hat\E(z)=I+(1-z)\hat\E(z)$ and since $(I+A)^{1/2}$ 
commutes with~$(I+A)\hat\E(z)$, we have 
$\|(I+A)\hat\E(z)v\|_{H^q}\le C|1-z||z|^{-1}\|v\|_{H^q}$ for any $q\in\R$,
and thus by Corollary~\ref{cor: I-Ph},
\[
\|(P_h-I)\hat\E(z)v\|\le Ch^\nu\|\hat\E(z)v\|_{H^{\nu}}
	=Ch^\nu\|(I+A)\hat\E(z)v\|_{H^{\nu-2}}
	\le Ch^\nu\|v\|_{H^{\nu-2}},
\]
noting that $|1-z||z|^{-1}\le C_{\omega,\beta}$ for $z\in\Sigma^\omega_\beta$.

To estimate the second term in~\eqref{eq: Gh split}, we write
\begin{align*}
\hat\E_h(z)P_h-P_h\hat\E(z)&=\hat\E_h(z)P_h(zI+A)\hat\E(z)
	-\hat\E_h(z)(zI+A_h)P_h\hat\E(z)\\
	&=\hat\E_h(z)[P_hA-A_hP_h]\hat\E(z).
\end{align*}
For all $u$, $w\in H^1$,
\begin{align*}
\iprod{P_h(I+A)u,w}&=\iprod{(I+A)u,P_hw}=a_1(u,P_hw)=a_1(R_hu,P_hw)\\
	&=\iprod{(I+A_h)R_hu,P_hw}=\iprod{(I+A_h)R_hu,w},
\end{align*}
so $P_h(I+A)=(I+A_h)R_h$ and thus 
\[
\hat\E_h(z)P_h-P_h\hat\E(z)=\hat\E_h(z)(I+A_h)P_h(R_h-I)\hat\E(z).
\]
Since $\hat\E_h(z)(I+A_h)=I+(1-z)\hat\E_h(z)$ the resolvent 
estimate~\eqref{equ:res est} and Theorem~\ref{thm: Rh error}
imply that
\begin{multline*}
\bigl\|\bigl[\hat\E_h(z)P_h-P_h\hat\E(z)\bigr]v\bigr\|
	\le\bigl(1+C|1-z||z|^{-1}\bigr)
	\|(R_h-I)\hat\E(z)v\|\\
	\le Ch^\nu\|\hat\E(z)v\|_{H^\nu}
	= Ch^\nu\|(I+A)\hat\E(z)v\|_{H^{\nu-2}}
	\le Ch^\nu\|v\|_{H^{\nu-2}},
\end{multline*}
noting again that $|1-z||z|^{-1}\le C_{\omega,\beta}$ 
for~$z\in\Sigma^\omega_\beta$.
\end{proof}

\begin{theorem}\label{thm:space dis}
Let $u$ be the solution of~\eqref{equ:heat} and let $u_h$ be the
semidiscrete approximation given by~\eqref{eq: uh weak}.
If $0\le\nu\le2\tau$, then
\[
\|u_h(t) - u(t)\|\le C h^\nu_X t^{-1}e^{\omega t}\bigl(
	\|u_0\|_{H^{\nu-2}}
	+\|\hat f\|_{H^{\nu-2},\partial\Sigma_{\beta}^{\omega}}
	\bigr), \quad \text{for $t>0$.}
\] 
\end{theorem}
\begin{proof}
Let $\Gamma_\pm$ be the half-line~$z=\omega+se^{\pm i\beta}$
for~$0<s<\infty$, so that $\Gamma=\Gamma_+-\Gamma_-$.  Since 
$\Re z=\omega-cs$ where~$c=-\cos\beta>0$, by applying 
Lemma~\ref{lem:Ghv} we have
\begin{align*}
\biggl\|\int_{\Gamma_\pm}e^{zt}G_h(z)g(z)\,dz\biggr\|
	&\le\int_0^\infty e^{(\omega-cs)t}\|G_h(z)g(z)\|\,ds\\
	&\le Ce^{\omega t}h^\nu\|g\|_{H^{\nu-2},\Gamma}\int_0^\infty
	e^{-cst}\,ds,
\end{align*}
and the error bound follows at once from the integral 
representation~\eqref{eq: uh-u integral}.
\end{proof}

Combining Theorems \ref{thm:time dis}~and \ref{thm:space dis}, we conclude
that provided $u_0$ and $f$ have the appropriate spatial regularity,
\begin{equation}\label{eq: U_N,h error}
\|U_{N,h}(t)-u(t)\|=O\bigl(\lg(\rho_rN)e^{-\mu N}+h_X^{2\tau}\bigr)
	\quad\text{for $T/2\le t\le 2T$,}
\end{equation}
where the constant includes a factor~$(1+T^{-1})e^{2\omega T}$.
Moreover, in the next section (Theorem~\ref{thm: nonsmooth}, Part~2)
we will see that when~$f\equiv0$ the error bound~\eqref{eq: U_N,h error}
remains valid even if the initial data is not regular.
\subsection{Nonsmooth initial data}
Consider the case $f\equiv0$, that is,
\begin{equation}\label{equ:homo}
\partial_tu-\LB u=0\quad\text{on $\Sphere^n$ for $t>0$,}
	\quad\text{with $u=u_0$ when $t=0$,}
\end{equation}
and the corresponding semidiscrete problem in which $u_h:[0,\infty)\to S_h$
satisfies
\begin{equation}\label{eq: semi homo}
\partial_tu_h-\LB_h u_h=0\quad\text{on $\Sphere^n$ for $t>0$,}
	\quad\text{with $u=u_{0h}$ when $t=0$,}
\end{equation}
where $\LB_h:S_h\to S_h$ is defined by
\[
\iprod{-\LB_h\psi,\chi}=a(\psi,\chi)=\iprod{\grad\psi,\grad\chi}
	\quad\text{for all $\psi$, $\chi\in S_h$;}
\]
compare with~\eqref{eq: Ah}.  In contrast to the forgoing analysis,
we now permit the initial data~$u_0$ to be an arbitrary function
in~$L_2(\Sphere^n)$.

By separating variables, we obtain an expansion in spherical harmonics,
\begin{equation}\label{eq: E(t)u0}
u(t) = \E(t)u_0 = 
\sum_{\ell=0}^\infty\sum_{k=1}^{N(n,\ell)} e^{-\lambda_\ell t} 
	\widehat{(u_0)}_{\ell,k} Y_{\ell k},
\end{equation}
that implies the smoothing property in the next theorem.

\begin{theorem}\label{thm: regularity}
Let $0\le q\le\nu$ and $m\in\{0,1,2,\ldots\}$.  If $u_0\in H^s$ then
$\E(t)u_0\in H^\nu$ and
\[
\|\partial^m_t \E(t) v\|_{H^\nu} \le C_T t^{-(\nu-q)/2-m}\|v\|_{H^q}, 
	\quad\text{for $0<t\le T$.}
\]
\end{theorem}
\begin{proof}
Adapting the argument of Thom\'ee~\cite[Lemma 3.2]{Tho97}, we see
from~\eqref{eq: E(t)u0} that the generalized Fourier coefficients
of~$\partial_t^m\E(t)u_0$ are
\[
\iprod{\partial_t^m\E(t)u_0,Y_{\ell k}}
	=(-\lambda_\ell)^me^{-\lambda_\ell t}\widehat{(u_0)}_{\ell k},
\]
so by~\eqref{eq: H^sigma norm},
\[
\|\partial_t^m\E(t)u_0\|_{H^\nu}^2=\sum_{\ell=0}^\infty
	(1+\lambda_\ell)^\nu\lambda_\ell^{2m}e^{-2\lambda_\ell t}
	\sum_{k=1}^{N(\ell,n)}\bigl|\widehat{(u_0)}_{\ell k}\bigr|^2.
\]
The result follows because, with~$s=\lambda_\ell t$,
\[
t^{\nu-q+2m}(1+\lambda_\ell)^{\nu-q}\lambda_\ell^{2m}e^{-2\lambda_\ell t}
	\le(T+s)^{\nu-q}s^{2m}e^{-2s}\le C_T
	\quad\text{for $0\le t\le T$.}
\]
\end{proof}

Let $\T:L_2\to H^2$ be the solution operator for the elliptic problem
\[
u-\LB u=f\quad\text{on $\Sphere^n$,}
\]
that is, $\T f:=u$.  Thus,
\[
a_1(\T f,v)=\iprod{f,v}\quad\text{for all $v\in H^1$,}
\]
and we can define $\T_h:L_2\to S_h$ by
\[
a_1(\T_hf,\chi)=\iprod{f,\chi}\quad\text{for all $\chi\in S_h$.}
\]
It follows that $\T_hf=R_hu=R_h\T f$ and $R_h=\T_h(I-\LB)$.  Since
\[
\iprod{f,\T_hw}=a_1(\T_hf,\T_hw)\quad\text{for all $f$, $v\in L_2$,}
\]
we see that $\T_h$ is self-adjoint and (taking $w=f$) strictly
positive-definite.

Rewriting the homogeneous equation~\eqref{equ:homo} as
$\partial_tu+(I-\LB)u=u$, we see that
\[
\T\partial_tu+u=\T u\quad\text{for $t>0$,}\quad\text{with $u(0)=u_0$,}
\]
and similarly the corresponding semidiscrete 
problem~\eqref{eq: semi homo} is equivalent to
\[
\T_h\partial_tu_h+u_h=\T_hu_h\quad\text{for $t>0$,}
	\quad\text{with $u_h(0)=u_{0h}$.}
\]
Thus, the error $\err=u_h-u$ satisfies
\begin{equation}\label{eq: e rho}
\T_h\partial_t\err+\err=\T_h\err+\rho\quad\text{where}\quad
\rho=(R_h-I)u.
\end{equation}

\begin{lemma}\label{lem: e rho}
With the notation above, if $u_{0h}=P_hu_0$ then
\[
\|\err(t)\|^2\le C_T\biggl(\|\rho(t)\|^2+\frac{1}{t}\int_0^t
	\bigl(s^2\|\partial_s\rho\|^2+\|\rho(s)\|^2\bigr)\,ds
	\biggr) \quad\text{for $0<t\le T$.}
\]
\end{lemma}
\begin{proof}
We modify the argument of Thom\'ee~\cite[Lemma~3.3]{Tho97}.
Taking the inner product of~\eqref{eq: e rho} with~$\partial_t\err$ 
gives
\[
\iprod{\T_h\partial_t\err,\partial_t\err}+\iprod{\err,\partial_t\err}
	=\iprod{\T_h\err+\rho,\partial_t\err},
\]
and since $\iprod{\T_h\partial_t\err,\partial_t\err}\ge0$ and
$\iprod{\err,\partial_t\err}=(1/2)\partial_t\|\err\|^2$, it follows that
\[
\partial_t\|\err\|^2\le2\iprod{\T_h\err+\rho,\partial_t\err},
\]
implying that
\[
\partial_t\bigl(t\|\err\|^2\bigr)=\|\err\|^2+t\partial_t\|\err\|^2
	\le\|\err\|^2+2t\iprod{\T_h\err+\rho,\partial_t\err}.
\]
Since 
\[
2t\iprod{\T_he,\partial_te}=t\partial_t\iprod{\T_he,e}
	\le\partial_t\bigl(t\iprod{\T_he,e}\bigr)
\]
and
\[
t\iprod{\rho,\partial_t\err}=
\partial_t\bigl(t\iprod{\rho,\err}\bigr)-t\iprod{\partial_t\rho,\err}
-\iprod{\rho,\err},
\]
we have
\[
\partial_t\bigl(t\|\err\|^2\bigr)\le\|\err\|^2
	+\partial_t\bigl(t\iprod{\T_h\err+2\rho,\err}\bigr)
	-2t\iprod{\partial_t\rho,\err}-2\iprod{\rho,\err},
\]
so integration gives
\[
t\|\err\|^2\le\int_0^t\|\err(s)\|^2\,ds+t\iprod{\T_h\err+2\rho,\err}
	+2\int_0^t\bigl|\iprod{s\partial_s\rho+\rho(s),\err(s)}\bigr|\,ds,
\]
and using $2\iprod{\rho,\err}\le4\|\rho\|^2+(1/2)\|\err\|^2$, 
\begin{equation}\label{eq: t||err||^2}
t\|\err\|^2\le2t\iprod{\T_h\err,\err}+8t\|\rho\|^2
	+2\int_0^t\bigl(s^2\|\partial_s\rho\|^2+\|\rho(s)\|^2
		+2\|\err(s)\|^2\bigr)\,ds.
\end{equation}
To deal with the terms in~$\err$ on the right-hand side,
take the inner product of~\eqref{eq: e rho} with~$\err$, obtaining
\[
(1/2)\partial_t\iprod{\T_h\err,\err}+\|\err\|^2=\iprod{\T_h\err+\rho,\err},
\]
or equivalently,
$\partial_t\iprod{\T_h\err,\err}-2\iprod{\T_h\err,\err}+2\|\err\|^2
=2\iprod{\rho,\err}$. After multiplying by the integrating 
factor~$e^{-2t}$,
\begin{equation}\label{eq: ifact}
\partial_t\bigl(e^{-2t}\iprod{\T_h\err,\err}\bigr)+2e^{-2t}\|\err\|^2
	=2e^{-2t}\iprod{\rho,\err},
\end{equation}
and the choice $u_{0h}=P_hu_0$ means that $\T_h\err(0)=0$ because
\[
\iprod{\T_he(0),w}=\iprod{\T_h(P_h-I)u_0,w}=\iprod{(P_h-I)u_0,\T_hw}=0
\]
for every $w\in L_2$.  Thus, 
\begin{multline*}
e^{-2t}\iprod{\T_h\err,\err}+2\int_0^t e^{-2s}\|\err(s)\|^2\,ds
	=2\int_0^t e^{-2s}\iprod{\rho(s),\err(s)}\,ds\\
	\le\int_0^te^{-2s}\bigl(\|\rho(s)\|^2+\|\err(s)\|^2\bigr)\,ds,
\end{multline*}
implying that
\[
\iprod{\T_h\err,\err}+\int_0^te^{2(t-s)}\|\err(s)\|^2\,ds
	\le\int_0^te^{2(t-s)}\|\rho(s)\|^2\,ds.
\]
Hence,
\[
2t\iprod{\T_h\err,\err}+4\int_0^t\|\err(s)\|^2\,ds
	\le2\max(t,2)\int_0^te^{2(t-s)}\|\rho(s)\|^2\,ds,
\]
and inserting this bound in~\eqref{eq: t||err||^2} gives 
\[
\|\err(t)\|^2\le8\|\rho(t)\|^2
	+\frac{2}{t}\int_0^ts^2\|\partial_s\rho\|^2\,ds
	+3\max(1,2t^{-1})\int_0^te^{2(t-s)}\|\rho(s)\|^2\,ds.
\]
\end{proof}

\begin{theorem}\label{thm: nonsmooth}
Let $u$ be the solution of the homogeneous problem~\eqref{equ:homo}
with initial data~$u_0$, let $u_h$ be the
semidiscrete approximation given by~\eqref{eq: semi homo} 
with~$u_{0h}=P_hu_0$. For $1\le\nu\le2\tau$:
\begin{enumerate}
\item if $u_0\in H^\nu(\Sphere^n)$, then
\[
\|u_h(t)-u(t)\|\le C_Th_X^\nu\|u_0\|_{H^\nu}
	\quad\text{for $0\le t\le T$;}
\]
\item if $u_0\in L_2(\Sphere^n)$ and $2\tau$ is an integer, then
\[
\|u_h(t)-u(t)\|\le C_Tt^{-\nu/2}h_X^\nu\|u_0\|
	\quad\text{for $0<t\le T$.}
\]
\end{enumerate}
\end{theorem}
\begin{proof}
We see at once from Lemma~\ref{lem: e rho} that
\[
\|\err(t)\|\le C_t\sup_{0\le s\le t}\bigl(\|\rho(s)\|
	+s\|\partial_s\rho\|\bigr),
\]
and if~$u_0\in H^\nu$ then, by Theorems \ref{thm: Rh error}~and
\ref{thm: regularity},
\[
\|\rho(s)\|+s\|\partial_s\rho(s)\|\le Ch^\nu\bigl(\|u(s)\|_{H^\nu}
	+s\|\partial_su(s)\|_{H^\nu}\bigr)
	\le Ch^\nu\|u_0\|_{H^\nu},
\]
which proves Part~1.

Assume now that $u_0\in L_2$.
By Theorems \ref{thm: Rh error}~and \ref{thm: regularity},
\[
\|\rho(t)\|=\|u(t)-R_hu(t)\|\le Ch\|u(t)\|_{H^1}
	\le Cht^{-1/2}\|u_0\|,
\]
and the expansion~\eqref{eq: E(t)u0} in spherical harmonics implies that
\begin{align*}
\int_0^t\|\rho(s)\|^2\,ds
	&\le Ch^2\int_0^t\|u(s)\|_{H^1}^2\,ds\\
	&=Ch^2\sum_{\ell=0}^\infty (1+\lambda_\ell)
	\sum_{k=1}^{N(n,\ell)}\bigl|\widehat{(u_0)}_{\ell k}\bigr|^2
	\int_0^te^{-2\lambda_\ell s}\,ds.
\end{align*}
If $\ell\ge1$ then $\lambda_\ell\ge\lambda_1=n$ so the
substitution~$s=\sigma/\lambda_\ell$ gives
\[
(1+\lambda_\ell)\int_0^te^{-2\lambda_\ell s}\,ds
	=\frac{1+\lambda_\ell}{\lambda_\ell}
	\int_0^{\lambda_\ell t}e^{-2\sigma}\,d\sigma
	\le(1+n^{-1})\int_0^\infty e^{-2\sigma}\,d\sigma\le1,
\]
and thus
\[
\int_0^t\|\rho(s)\|^2\,ds\le Ch^2\biggl(t|\widehat{(u_0)}_{01}|^2
	+\sum_{\ell=1}^\infty\sum_{k=1}^{N(n,\ell)}
		|\widehat{(u_0)}_{\ell k}|^2\biggr)\le C_th^2\|u_0\|^2.
\]
Similarly, 
\begin{align*}
\int_0^ts^2\|\partial_s\rho\|^2\,ds
	&\le Ch^2\int_0^t s^2\|\partial_su(s)\|_{H^1}^2\,ds\\
	&=Ch^2\sum_{\ell=1}^\infty (1+\lambda_\ell)\lambda_\ell^2 
	\int_0^t s^2e^{-2\lambda_\ell s}\,ds
	\sum_{k=1}^{N(n,\ell)}
	|\widehat{(u_0)}_{\ell k}|^2
\end{align*}
and for all $\ell\ge1$,
\[
(1+\lambda_\ell)\lambda_\ell^2\int_0^t s^2e^{-2\lambda_\ell s}\,ds
	=\frac{1+\lambda_\ell}{\lambda_\ell}\int_0^{\lambda_\ell t}
	\sigma^2e^{-2\sigma}\,d\sigma
	\le C,
\]
so $\int_0^ts^2\|\partial_s\rho\|^2\,ds\le Ch^2\|u_0\|^2$.  Applying
Lemma~\ref{lem: e rho}, Part~2 follows in the special case~$\nu=1$.

To deal with case~$\nu=2\tau$, we introduce the
solution operator for the semidiscrete problem, $\E_h(t)u_0:=u_h(t)$,
and use the semigroup property: $\E(s+t)=\E(s)\E(t)$~and 
$\E_h(s+t)=\E_h(s)\E_h(t)$ for all $s$~and $t$.  The error
operator~$\F_h(t)=\E_h(t)-\E(t)$ satisfies the identity
\begin{align*}
\F_h(t)-\F_h(t/2)^2&=\E_h(t/2)^2-\E(t/2)^2-\bigl[\E_h(t/2)-\E(t/2)\big]^2\\
	&=\F_h(t/2)\E(t/2)+\E(t/2)\F_h(t/2),
\end{align*}
and by Part~1 and Theorem~\ref{thm: regularity},
\[
\|\F_h(t/2)\E(t/2)u_0\|\le Ch^\nu\|\E(t/2)u_0\|_{H^\nu}
	\le Ch^\nu(t/2)^{-\nu/2}\|u_0\|.
\]
Since $\E(t/2)$~and $\F_h(t/2)$ are self-adjoint in~$L_2$, the
same estimate holds for the reversed product~$\E(t/2)\F_h(t/2)$,
and therefore
\begin{equation}\label{eq: Fh(t)}
\|\F_h(t)u_0\|\le Ct^{-\nu/2}h^\nu\|u_0\|+Ct^{-1/2}h\|\F_h(t/2)u_0\|.
\end{equation}
The stability estimates $\|\E(t)u_0\|\le\|u_0\|$~and 
$\|\E_h(t)u_0\|\le C\|u_0\|$ mean that it suffices to consider the
case $t^{-1/2}h\le1$, when repeated application of the 
estimate~\eqref{eq: Fh(t)} gives
\[
\|\F_h(t)u_0\|\le Ct^{-\nu/2}h^\nu\|u_0\|+C(t^{-1/2}h)^j\|\F_h(t/2^j)u_0\|
\]
for $j=0$, 1, 2, \dots, $\nu=2\tau$, and thus
$\|\F_h(t)u_0\|\le Ct^{-\tau}h^{2\tau}\|u_0\|$.  For the remaining 
case~$1<\nu<2\tau$, let $\theta=\nu/(2\tau)$ and observe that
\begin{align*}
\|\F_h(t)u_0\|&=\|\F_h(t)u_0\|^{1-\theta}\|\F_h(t)u_0\|^\theta\\
	&\le C\|u_0\|^{1-\theta}\bigl[(t^{-1/2}h)^{2\tau}
		\|u_0\|\bigr]^\theta=Ct^{-\nu/2}h^\nu\|u_0\|.
\end{align*}
\end{proof}
\section{Numerical experiments}\label{sec: num exp}
We present the results of some numerical experiments with two model
problems.  In both cases, the integration contour~\eqref{equ:z xi} and
quadrature step~$k$ are chosen as in Theorem~\ref{thm:time dis}, with
\[
T=1,\quad\omega=1,\quad\theta=1/2,\quad\delta=\pi/4,\quad r=\pi/4;
\]
Figure~\ref{fig:cont} shows the case~$N=20$.  Our conference 
paper~\cite{LeGiaMcLean2011} presents some earlier numerical examples.

\subsection{A scalar problem}
Consider the ODE $u'+u=f(t)$ for~$t>0$, with $u(0)=1$.  We
choose the source term~$f$ so that the exact solution is
\[
u(t) = 1+ \frac{4t^{3/2}} {3\sqrt{\pi}},
\] 
which has the Laplace transform~$\hat u(z) = z^{-1} + z^{-5/2}$.
In this case, no spatial discretization is required, and the
numerical solution~$U_N$ is given by~\eqref{equ:UN}.
Table~\ref{tab:err1} shows the error at~$t=2$ for different
values of~$N$.  The rapid convergence is consistent with the
error bound of Theorem~\ref{thm:time dis}, but as~$N$ increases
the quadrature eventually becomes unstable.

\begin{figure}
\begin{center}
\scalebox{0.6}{\includegraphics{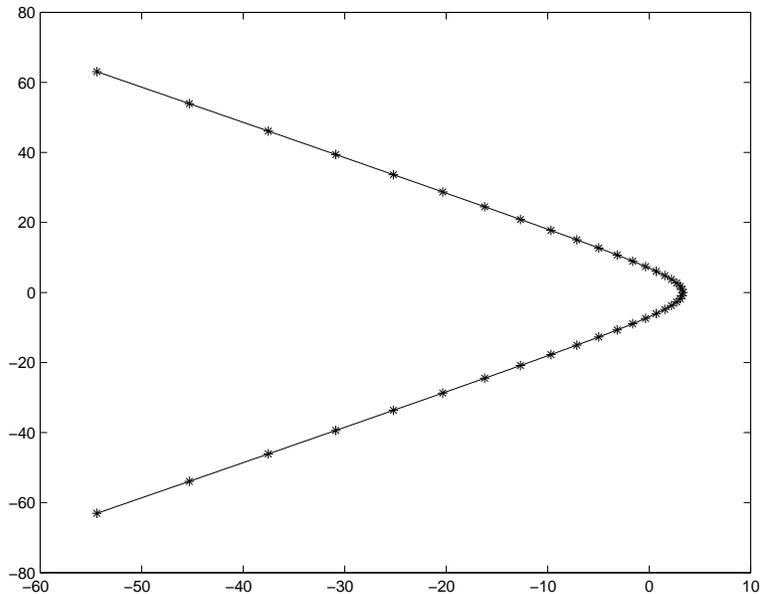}}
\caption{The integration contour~$\Gamma$ and quadrature points~$z_j$
when~$N=20$.}\label{fig:cont}
\end{center}
\end{figure}
\begin{table}[h]
\begin{center}
\renewcommand{\arraystretch}{1.2}
\begin{tabular}{c|ccccc}
\hline
$N$            &    10&        20&        30&        35&        40 \\ 
\hline
$|U_N(2)-u(2)|$&1.71E-04 & 6.44E-08 & 3.75E-11& 7.52E-13 &1.16E-12 \\

\hline
\end{tabular}
\caption{Errors for a scalar problem.}\label{tab:err1}
\end{center}
\end{table}
\subsection{Heat equation on the unit sphere}
Fix $0<a<1$ and define $u_0:\Sphere^2\to\C$ for $x=(x_1,x_2,x_3)\in\Sphere^2$
by
\begin{equation}\label{equ:u0}
u_0(x)=\begin{cases}
	1,&\text{if $a\le x_3\le 1$,}\\
	0,&\text{if $-1\le x_3<a$.}
\end{cases}
\end{equation}
This axially symmetric function has the Fourier--Legendre expansion
\[
u_0(x)=\sum_{\ell=0}^\infty \widehat{(u_0)}_\ell P_\ell(x_3),
\quad\text{where $\widehat{(u_0)}_\ell
	 =\frac{2\ell+1}{2}\int_a^1 P_\ell(t) dt$.}
\]   
The zeroth coefficient is $\widehat{(u_0)}_0=(1-a)/2$, and the remaining 
coefficients are expressible in terms of Jacobi polynomials 
\cite[page~172]{Bye59}, \cite[Formula~18.9.15]{dlmf11}, 
\[   
\widehat{(u_0)}_\ell=\frac{1-a^2}{2}\,\frac{2\ell+1}{\ell(\ell+1)}
	P_\ell'(a)
	=(1-a^2) \frac{(2\ell+1)} {4\ell}\,P^{(1,1)}_{\ell-1}(a)
	\quad\text{for $\ell\ge1$;}
\]
consequently $\widehat{(u_0)}_\ell=O(\ell^{-1/2})$ as~$\ell\to\infty$
\cite[Theorem 7.32.2]{Sze59}.

The PDE~$u_t-\LB u=0$ with initial data~\eqref{equ:u0} describes heat 
diffusion from a spherical cap about the north pole onto the surface 
of the unit sphere~$\Sphere^2$.  By separating variables, we find 
that the exact solution is
\[
u(x,t) = \sum_{\ell=0}^\infty e^{-\ell(\ell+1)t}\widehat{(u_0)}_\ell
	P_\ell(x_3), \quad\text{for $x = (x_1,x_2,x_3)\in\Sphere^2$.}
\]

\begin{table}
\begin{center}
\renewcommand{\arraystretch}{1.2}
\begin{tabular}{|c|c|c|c|}
\hline
$m$ & $\rho_m(r)$    & Smoothness &$\tau$ \\
 \hline
 $2$ & $(1-r)^6_{+}(3+ 18r + 35r^2)$       & $C^4$ & $7/2$\\
 $3$ & $(1-r)^8_{+}(1+ 8r + 25r^2+ 32r^3)$ & $C^6$ & $9/2$\\
\hline
\end{tabular}
\caption{The compactly supported SRBFs of Wendland~\cite{Wen05}.}
\label{tab:rbfs}
\end{center}
\end{table}

For the spatial discretization, we use the compactly supported radial
basis functions introduced by Wendland~\cite{Wen05}, for which the
strictly positive-definite kernel has the form
\[
\Phi(x,y) = \rho_m\bigl(\sqrt{2-2x \cdot y}\bigr).
\]
In Table~\ref{tab:rbfs}, we show $\rho_2$~and $\rho_3$ explicitly,
along with the values of the exponent~$\tau$ in~\eqref{cond:what}.
We generate the set of points~$X$ using an equal area partitioning
algorithm of Saff and Kuijlaars~\cite{SafKui97}.  To compute the inner
products arising in the matrix
entries~\eqref{equ:B S} and the load vector components~$G_p(z)$, we use
a quadrature approximation of the form
\begin{equation}\label{equ:sphere quad}
\int_{\Sphere^2}v\,dS\approx\frac{2\pi}{R}\sum_{q=1}^R\sum_{p=1}^{R/2}w_p
v\bigl(\sin\theta_p\cos\phi_q,\sin\theta_p\sin\phi_q,\cos\theta_p\bigr),
\end{equation}
for an even number~$R\ge2$,
where $\int_{-1}^1f(z)\,dz\approx\sum_{p=1}^{R/2}w_pf(\cos\theta_p)$
is a Gauss--Legendre rule and $\phi_q=2\pi q/R$.  The error in the
approximation~\eqref{equ:sphere quad} is zero if the integrand~$v$ is
a polynomial of total degree~$R-1$ or less.


\begin{table}
\begin{center}
\begin{tabular}{|c|c|c|c|c|c|c|}
\hline
& $K$        &      200     &     400     &   600        &  801      &   1001     \\
& $h_X$      &    0.1796    &   0.1281    &   0.1039     & 0.0888    &  0.0794    \\
&  $R$       &     200      &    200      &  200         &  500      &    500     \\
\hline
$N=10$ & $e_{\max}$ &5.67E-05 & 5.06E-06 & 1.67E-06 & 9.89E-07 & 8.24E-07 \\ 
       & $e_2$      &4.63E-05 & 3.60E-06 & 1.10E-06 & 8.23E-07 & 7.81E-07 \\ 
       & EOC($e_2$) &         & 7.56E+00 & 5.66E+00 & 1.84E+00 & 4.66E-01 \\ 
\hline
$N=20$ & $e_{\max}$ &5.61E-05 & 4.47E-06 & 1.03E-06 & 3.26E-07 & 1.48E-07 \\ 
       & $e_2$      &4.61E-05 & 3.53E-06 & 8.07E-07 & 2.64E-07 & 1.20E-07 \\ 
       & EOC($e_2$) &         & 7.60E+00 & 7.05E+00 & 7.11E+00 & 7.06E+00 \\
\hline 
$N=30$ & $e_{\max}$ &5.61E-05 & 4.47E-06 & 1.03E-06 & 3.26E-07 & 1.48E-07 \\ 
       & $e_2$      &4.61E-05 & 3.53E-06 & 8.07E-07 & 2.64E-07 & 1.20E-07 \\ 
       & EOC($e_2$) &         & 7.60E+00 & 7.05E+00 & 7.11E+00 & 7.06E+00 \\ 
\hline
$N=35$ & $e_{\max}$ &5.61E-05 & 4.47E-06 & 1.03E-06 & 3.26E-07 & 1.48E-07 \\ 
       & $e_2$      &4.61E-05 & 3.53E-06 & 8.07E-07 & 2.64E-07 & 1.20E-07 \\ 
       & EOC($e_2$) &         & 7.60E+00 & 7.05E+00 & 7.11E+00 & 7.06E+00 \\
\hline
\end{tabular}
\caption{Numerical results with SRBFs constructed using $\rho_2$.} 
\label{tab:RBF4}
\end{center}
\end{table}
\begin{table}
\begin{center}
\begin{tabular}{|c|c|c|c|c|c|c|}
\hline
       & $K$        &      200  &     400   &   600      &  801    &   1001   \\
       &  $h_X$     &    0.1796 &   0.1281  &   0.1039   & 0.0888  &  0.0794  \\
       &  $R$       &     200   &    200    &  200       &  500    &    500     \\
\hline
$N=10$ & $e_{\max}$ & 6.86E-05 & 3.84E-06 & 1.17E-06 & 8.35E-07 & 7.79E-07 \\ 
       & $e_2$      & 3.93E-05 & 1.63E-06 & 7.85E-07 & 7.74E-07 & 7.71E-07 \\ 
       & EOC($e_2$) &          & 9.41E+00 & 3.50E+00 & 9.13E-02 & 3.37E-02 \\ 
\hline
$N=20$ & $e_{\max}$ & 6.78E-05 & 3.11E-06 & 4.54E-07 & 8.98E-08 & 3.24E-08 \\ 
       & $e_2$      & 3.91E-05 & 1.45E-06 & 2.12E-07 & 4.83E-08 & 1.73E-08 \\ 
       & EOC($e_2$) &          & 9.75E+00 & 9.17E+00 & 9.41E+00 & 9.18E+00 \\ 
\hline
 $N=30$& $e_{\max}$ & 6.78E-05 & 3.11E-06 & 4.54E-07 & 8.98E-08 & 3.24E-08 \\ 
       & $e_2$      & 3.91E-05 & 1.45E-06 & 2.12E-07 & 4.83E-08 & 1.73E-08 \\ 
       & EOC($e_2$) &          & 9.75E+00 & 9.17E+00 & 9.41E+00 & 9.18E+00 \\ 
\hline
$N=35$ & $e_{\max}$ & 6.78E-05 & 3.11E-06 & 4.54E-07 & 8.98E-08 & 3.24E-08 \\ 
       & $e_2$      & 3.91E-05 & 1.45E-06 & 2.12E-07 & 4.83E-08 & 1.73E-08 \\ 
       & EOC($e_2$) &          & 9.75E+00 & 9.17E+00 & 9.41E+00 & 9.18E+00 \\
\hline 
\end{tabular}
\caption{Numerical results with SRBFs constructed using $\rho_3$.} 
\label{tab:RBF6}
\end{center}
\end{table}

In the numerical experiments, we let $a = 0.9$ in the 
definition~\eqref{equ:u0} of~$u_0$. Tables~\ref{tab:RBF4}~and 
\ref{tab:RBF6} show values of the quantities
\[
 e_{\max} = \max_{x \in \QQ}\bigl|U_{N,h}(x,1) - u(x,1)\bigr|
\]
and
\[
 e_2 = \biggl(\sum_{x \in \QQ} w_{x} \bigl|U_{N,h}(x,1) - u(x,1)\bigr|^2
	\biggr)^{1/2},
\]
for different choices of $K$~and $R$.  Here, $\QQ$ is the set of 
quadrature points.

Since $u_0\in L_2(\Sphere^2)$, we expect from 
Theorem~\ref{thm: nonsmooth} and the triangle 
inequality~\eqref{eq: triangle ineq} that if~$N$ is sufficiently large 
then $e_2=O(h^{2\tau})$ --- that is, $O(h^7)$ using $\rho_2$, and $O(h^9)$ 
using $\rho_3$.  The observed convergence rates are close to these 
predicted values.
We remark that when~$K=1001$, the condition number of the linear 
system~\eqref{eq: uh hat} is around $10^7$ using $\rho_2$, and around 
$10^9$ using $\rho_3$, so we cannot expect to reduce the error much
below the smallest values shown in the tables.

\paragraph*{Acknowledgement} The first author is supported by the Australian Research Council.
\bibliographystyle{plain}
\bibliography{legia_mclean-refs}
\end{document}